%% file: main.tex
\documentclass[reqno]{amsart}

\input{Preamble}

\newcommand{\abs}[1]{\lvert#1\rvert}
\crefname{property}{property}{properties}
\newcommand{\cutsep}{cut-sep\-a\-ra\-tion}
\newcommand{\bondsep}{bond-sep\-a\-ra\-tion}
\newcommand{\tcd}{tree-cut de\-com\-po\-si\-tion}

\begin{document}
\vspace*{-1.14cm} 
\title[Edge-connectivity and tree-structure in finite and infinite graphs]{Edge-connectivity and tree-structure\\in finite and infinite graphs}

\author{Christian Elbracht}
\address{Universität Hamburg, Department of Mathematics, Bundesstraße 55 (Geomatikum), 20146 Hamburg, Germany}
\email{$\{$christian.elbracht, jan.kurkofka, maximilian.teegen$\}$@uni-hamburg.de}

\author{Jan Kurkofka}

\author{Maximilian Teegen}

\@namedef{subjclassname@2020}{\textup{2020} Mathematics Subject Classification}

\begin{abstract}
We show that every graph admits a canonical tree-like decomposition into its $k$-edge-connected pieces for all $k\in\N\cup\{\infty\}$ simultaneously.
\end{abstract}

\maketitle

\vspace{-.7cm}
\section{Introduction}

\noindent Finding a tree-like decomposition of any finite graph into its `$k$-vertex-connected pieces', for just one given $k\in\N$ or all $k\in\N$ at once, has been a longstanding quest in graph theory until recently, when it was solved comprehensively by Diestel, Hundertmark and Lemanczyk~\cite{ProfilesNew}. 
One of the complications was that there are many competing notions of what a `$k$-vertex-connected piece' of a graph should be.
Instead of providing a dozen independent solutions for the dozen different notions of `$k$-vertex-connected pieces' that are in use, the solution in~\cite{ProfilesNew} deals with all these notions at once.
Related results can be found in~\cites{short,kBlocks,CarmesinOne,CarmesinTwo,ConnectivityTreeStructure,StructuralTT,carmesin2020canonical,BlockTangleDuality,TTDAbstract,TTDgraphs,MonaLisa,ProfilesNew,VertexCuts,elbracht2020canonical,FiniteSplinters,InfiniteSplinters,elm2020treeoftangles,Refining,Matroids,GeelenGridMatroid,Grohe,Valentin,GMX,Tutte}.

If we consider edge-connectivity instead of vertex-connectivity, however, there does exist a single notion of `$k$-edge-connected pieces' that undeniably is the most natural one.
Let $k\in\N\cup\{\infty\}$ and let $G$ be any connected graph, possibly infinite. 
We say that two vertices or ends are (${<}k$)-\emph{inseparable} in $G$ if they cannot be separated in $G$ by fewer than $k$ edges.
This defines an equivalence relation on $\hat{V}(G):=V(G)\cup\Omega(G)$ where $\Omega(G)$ denotes the set of ends of~$G$ (which is empty if $G$ is finite).
Its equivalence classes are the `$k$-edge-connected pieces' of $G$, its \emph{$k$-edge-blocks}.
A subset of $\hat{V}(G)$ is an \emph{edge-block} if it is a $k$-edge-block for some~$k$.
Note that any two edge-blocks are either disjoint or one contains the other.
In this paper we find a canonical tree-like decomposition of any connected graph, finite or infinite, into its $k$-edge-blocks---for all $k\in\N\cup\{\infty\}$ simultaneously. 
To state our result, we only need a few intuitive definitions.

A subset $X\subset\hat{V}(G)$ \emph{lives in} a subgraph $C\subset G$ or vertex set $C\subset V(G)$ if all the vertices of $X$ lie in $C$ and all the rays of ends in $X$ have tails in $C$ or~$G[C]$, respectively.
If $G$ is finite, saying that $X$ lives in $C$ simply means that~$X\subset C$.
An edge set $F\subset E(G)$ \emph{distinguishes} two edge-blocks of $G$, not necessarily $k$-edge-blocks for the same~$k$, if they live in distinct components of $G-F$.
It distinguishes them \emph{efficiently} if they are not distinguished by any edge set of smaller size.
Note that if $F$ distinguishes two edge-blocks efficiently, then $F$ must be a \emph{bond}, a cut with connected sides.
A set $B$ of bonds \emph{distinguishes} some set of edge-blocks of~$G$ \emph{efficiently} if every two disjoint edge-blocks in this set are distinguished efficiently by a bond in~$B$.
Two cuts $F_1,F_2$ of $G$ are \emph{nested} if $F_1$ has a side $V_1$ and $F_2$ has a side $V_2$ such that $V_1\subset V_2$. Note that this is symmetric.
The fundamental cuts of a spanning tree, for example, are (pairwise) nested.
Our main result reads as follows:
\begin{mainresult}\label{thm:main}
Every connected graph $G$ has a nested set  
of bonds that efficiently distinguishes all the edge-blocks of~$G$. 
\end{mainresult}

\noindent The nested sets $N=N(G)$ that we construct, one for every~$G$, have two strong additional properties:
\begin{enumerate}[label=(\roman*)]
    \item They are canonical in that they are invariant under isomorphisms: if $\phi\colon G\to G'$ is a graph-isomorphism, then $\phi(N(G))=N(\phi(G))$.
    \smallskip
    \item For every $k\in\N$, the subset $N_k\subset N$ formed by the bonds of size less than $k$ is equal to the set of fundamental cuts of a tree-cut decomposition of~$G$ that decomposes $G$ into its $k$-edge-blocks.
    \smallskip
\end{enumerate}

\noindent \emph{Tree-cut decompositions} are decompositions of graphs similar to tree-decompositions but based on edge-cuts rather than vertex-separators.
They were introduced by Wollan~\cite{Wollan}, and they are more general than the `tree-partitions' introduced by Seese~\cite{Seese} and by Halin~\cite{HalinTreePartition}; see Section~\ref{sec:TreePartition}.

The second additional property above is best possible in the sense that $N_k$ cannot be replaced with~$N$:
there exists a graph $G$ (see Example~\ref{example:BigPicture}) that has no nested set of cuts which, on the one hand, distinguishes all the edge-blocks of~$G$ efficiently, and on the other hand, is the set of fundamental cuts of some tree-cut decomposition.
(This is because the `tree-structure' defined by a nested set of cuts may have limit points, and hence not be representable by a graph-theoretical tree.)

It turns out that the nested sets of bonds which make Theorem~\ref{thm:main} true can be characterised in terms of generating bonds (for the definition of \emph{generate} see Section~\ref{sec:DD}):

\begin{mainresult}\label{thm:generate}
Let $G$ be any connected graph and let $M$ be any nested set of bonds of~$G$. 
Then the following assertions are equivalent:
\begin{enumerate}
    \item $M$ efficiently distinguishes all the edge-blocks of~$G$;
    \item For every $k\in\N$, the ${\le} k$-sized bonds in~$M$ generate all the $k$-sized cuts of~$G$.
\end{enumerate}
\end{mainresult}

\noindent 
Nested sets of bonds which are canonical and satisfy assertion~(ii) of Theorem~\ref{thm:generate} have been constructed by Dicks and Dunwoody using their algebraic theory of graph symmetries.
This is one of the main results of their monograph~\cite{DD}*{II 2.20f}.
Since the implication (ii)$\to$(i) of Theorem~\ref{thm:generate} is straightforward, Theorem~\ref{thm:main} can be deduced from their theory, but it is not stated in~\cite{DD} explicitly.
Our Theorem~\ref{thm:generate} itself, in particular its highly non-trivial forward implication (i)$\to$(ii), does not follow from material in~\cite{DD}.
Since our proofs are purely combinatorial, we can combine Theorem~\ref{thm:main} and the forward implication (i)$\to$(ii) of Theorem~\ref{thm:generate} to obtain a purely combinatorial proof of the main result of Dicks and Dunwoody.
Together, our proofs of Theorem~\ref{thm:main} and Theorem~\ref{thm:generate} take just over~7 pages in total.

This paper is organised as follows.
In Section~\ref{sec:Prelim} we introduce the tools and terminology that we need.
In Section~\ref{sec:MainProof} we prove our main result, Theorem~\ref{thm:main}, and we show that we obtain a canonical set~$N$.
In Section~\ref{sec:TreePartition} we relate each $N_k$ to a tree-cut decomposition.
In Section~\ref{sec:DD} we prove Theorem~\ref{thm:generate}.
In Section~\ref{sec:infEdgeBlocks} we recall a theorem about spanning trees that distinguish all $\infty$-edge-blocks.

\section{Tools and terminology}\label{sec:Prelim}

\noindent We use the graph-theoretic notation of Diestel's book~\cite{DiestelBook5}.
Throughout this paper, $G=(V,E)$ denotes any connected graph, finite or infinite.
When we say ends we mean vertex-ends as usual, not edge-ends.
If a subset $X\subset\hat{V}(G)$, usually an edge-block, lives in a subgraph $C\subset G$ or vertex set $C\subset V(G)$, we denote this by $X\lives C$ for short.
Recall that $X\lives C$ defaults to $X\subset C$ if $G$ is finite.

The following lemma is well known \cite{DiestelBook5}*{Exercise~8.12}; we provide a proof for the reader's convenience.

\begin{lemma}\label{lem:finitely_many_bonds}
Every edge of a graph lies in only finitely many bonds of size $k$ of that graph, for any $k\in\N$.
\end{lemma}

\begin{proof}
Let $e$ be any edge of a graph $G$, and suppose for a contradiction that $e$ lies in infinitely many distinct bonds $B_0,B_1,\ldots$ of size~$k$, say.
Let $F$ be an inclusionwise maximal set of edges of $G$ such that $F$ is included in $B_n$ for infinitely many~$n$ (all~$n$, without loss of generality).
Then $\vert F\vert <k$ because the bonds are distinct, and any bond $B_n\supsetneq F$ gives rise to a path $P$ in $G-F$ that links the endvertices of~$e$.
Now all the infinitely many bonds $B_n$ must contain an edge of the finite path~$P$.
But by the choice of~$F$, each edge of $P$ lies in only finitely many $B_n$, a contradiction.
\end{proof}

\begin{corollary}\label{cor:noOmegaOrPlusOne}
Let $G$ be any connected graph, $k\in\N$, and let $F_0,F_1,\ldots$ be infinitely many distinct bonds of $G$ of size at most $k$ such that each bond $F_n$ has a side $A_n$ with $A_n\subsetneq A_m$ for all $n<m$.
Then $\bigcup_{n\in\N} A_n=V$.
\end{corollary}
\begin{proof}
If $\bigcup_n A_n$ is a proper subset of $V$, then any $A_0$--$(V\setminus\bigcup_n A_n)$ path in~$G$ admits an edge that lies in infinitely many $F_n$, contradicting Lemma~\ref{lem:finitely_many_bonds}.
\end{proof}

\subsection{Cuts, bonds and separations}

The \emph{order} of a cut is its size.
A \emph{\cutsep } of a graph $G$ is a bipartition $\{A,B\}$ of the vertex set of~$G$, and it \emph{induces} the cut $E(A,B)$.
Then the order of the cut $E(A,B)$ is also the \emph{order} of $\{A,B\}$.
Recall that in a connected graph, every cut is induced by a unique \cutsep\ in this way, to which it \emph{corresponds}.
A \emph{\bondsep } of $G$ is a \cutsep\ that induces a bond of~$G$, a cut with connected sides.
We say that a \cutsep\ \emph{distinguishes} two edge-blocks (\emph{efficiently}) if its corresponding cut does, and we call two \cutsep s \emph{nested} if their corresponding cuts are nested.
Thus, two \cutsep s $\{A,B\}$ and $\{C,D\}$ are nested if one of the four inclusions $A\subset C$, $A\subset D$, $B\subset C$ or $B\subset D$ holds.

\subsection{Key tool}\label{sec:keytool}

The proof of our main result relies on a result from~\cite{InfiniteSplinters}.
To state it, we shall need the following definitions.
Let $ \cA $ be some set and $ \sim $ a reflexive and symmetric binary relation on~$ \cA $. We say that two elements $ a $ and $ b $ of $ \cA $ are \emph{nested} if~$ a\sim b $ and two elements of $ \cA $ which are not nested \emph{cross}. A subset of $ \cA $ is called nested if its elements are pairwise nested.
In our setting, $\cA$ will be the set of all the \bondsep s of a connected graph $G$ that efficiently distinguish some edge-blocks of~$G$, and ${\sim}$ will encode `being nested' for \bondsep s.

Given $ a,b\in \cA$, we call $ c\in\cA $ a \emph{corner} of $ a $ and $ b $ if every element of~$ \cA $ which is nested with both $a$ and~$b$ is also nested with $c$.
When $a=\{A,B\}$ and $b=\{C,D\}$ are two \bondsep s, then $c$ will usually be one of the following four possible corners:
either $\{A\cap C,B\cup D\}$, $\{A\cap D,B\cup C\}$, $\{B\cap D,A\cup C\}$ or $\{B\cap C,A\cup D\}$.
These are the four possibilities of how a new \cutsep\ can be built from $\{A,B\}$ and $\{C,D\}$ using just `$\cup$' and~`$\cap$'. 
Note that sometimes an intersection may be empty so some of the four possibilities may not be valid \cutsep s; and sometimes a possibility is a \cutsep\ but not an element of~$\cA$.
We will see in Lemma~\ref{lem:fish} that every possibility that happens to lie in~$\cA$ is already a corner of $\{A,B\}$ and $\{C,D\}$, provided that $\{A,B\}$ and $\{C,D\}$ cross.

Consider a family $ (\,\cA_i\mid i\in I\,)$ of non-empty subsets of $ \cA $ and some function $ \abs{\,\cdot\,}\colon I\to\N $, where $ I $ is a possibly infinite index set. We call $ \abs{i} $ the \emph{order} of the elements of~$ \cA_i $. 
We will consider $I$ to be the collection of all the unordered pairs formed by two disjoint edge-blocks of~$G$, and each $\cA_i$ will consist of all the \bondsep s of~$G$ that efficiently distinguish the two edge-blocks forming the pair~$i$.
Then every $\cA_i$ will be non-empty because the edge-blocks forming~$i$ are disjoint.
Our choice for~$\abs{i}$ will be the unique natural number that is the order of all the \bondsep s in~$\cA_i$.
Note that each of the two edge-blocks forming~$i$ will be a $k$-edge-block for some~$k>\abs{i}$.

When we wish to prove Theorem~\ref{thm:main} without its additional properties, then it suffices to find a subset $N\subset\cA$ that meets each $\cA_i$ and that is nested.
One of the main results of~\cite{InfiniteSplinters} states that we can find $N$ if the setup of the sets $\cA_i$ and their order function $\abs{\,\cdot\,}$ satisfies a number of properties.
The result can be applied even when $I$ is infinite, and moreover it ensures that $N$ is `canonical' for the given setup.
To state the properties and the result, we need one more definition.

The $ k $-\emph{crossing number} of~$ a $, for an $ a\in\cA $ and $ k\in\N $, is the number of elements of~$ \cA $ that cross~$ a $ and lie in some~$ \cA_i $ with~$ \abs{i}=k $.
Note that in our case, every \bondsep\ of order~$k$ can only possibly lie in sets $\cA_i$ with $\abs{i}=k$.
Thus, the $k$-crossing number of a \bondsep\ of arbitrary finite order will be the number of efficiently distinguishing \bondsep s of order $k$ crossing it.

We say that the family $ (\,\cA_i\mid i\in I\,)$ \emph{thinly splinters} if it satisfies the following three properties:
\begin{enumerate}
	\item For every $ i\in I $ all elements of $ \cA_i $ have finite $ k $-crossing number for all~\mbox{$k\le\abs{i} $}.\label[property]{property:fin_cn}
	\item If $ a_i\in\cA_i $ and $ a_j\in\cA_j $ cross with $ \abs{i}<\abs{j} $, then $ \cA_j $ contains some corner of~$ a_i $ and~$ a_j $ that is nested with~$ a_i $. \label[property]{property:fish}
	\item If $ a_i\in\cA_i $ and $ a_j\in\cA_j $ cross with $ \abs{i}=\abs{j}=k \in \N $, then either $ \cA_i $ contains a corner of $ a_i $ and~$ a_j $ with strictly lower $ k $-crossing number than~$ a_i $, or else $ \cA_j $ contains a corner of $ a_i $ and $ a_j $ with strictly lower~$ k $-crossing number than~$ a_j $. \label[property]{property:strong_submodular}
\end{enumerate}
The following theorem from \cite{InfiniteSplinters}, whose proof takes little more than half a page, will be the key ingredient for our proof of \cref{thm:main}:
\begin{theorem}[\,\cite{InfiniteSplinters}*{Theorem 1.2}\,]\label{thm:thinly}
If  $(\,\cA_i \mid i\in I\,)$ thinly splinters with respect to some reflexive symmetric relation $\sim$ on $\cA:=\bigcup_{i\in I}\cA_i$, then there is a set $N=N((\,\cA_i\mid i\in I\,))\subseteq \cA$ which meets every $\cA_i$ and is nested, i.e., $n_1\sim n_2$ for all $n_1,n_2\in N$.
Moreover, this set $N$ can be chosen invariant under isomorphisms:  
if $ \phi $ is an isomorphism between $ (\cA,\sim)$ and $(\cA',\sim')$, then we have $ N((\,\phi(\cA_i)\mid i\in I\,))=\phi(N((\,\cA_i\mid i\in I\,))) $.
\end{theorem}

\section[Proof of the main result]{Proof of \texorpdfstring{\cref{thm:main}}{the main result}}\label{sec:MainProof}

\noindent Let $G$ be any connected graph, possibly infinite, and consider the set $\cA$ with the relation ${\sim}$ of `being nested', the family $(\,\cA_i\mid i\in I\,)$ and the function $\abs{\,\cdot\,}$, all defined with regard to the efficiently distinguishing \bondsep s of~$G$ like in Section~\ref{sec:keytool}.
Our aim is to employ Theorem~\ref{thm:thinly} to deduce Theorem~\ref{thm:main}.
In order to do that, we first have to verify that $(\,\cA_i\mid i\in I\,)$ thinly splinters.
To this end, we verify all the three properties~\ref{property:fin_cn}--\ref{property:strong_submodular} below.
The following lemma clearly implies~\cref{property:fin_cn}:
\begin{lemma}\label{lem:splinter_1}
Every finite-order \bondsep\ of a graph $G$ is crossed by only finitely many \bondsep s of~$G$ of order at most $k$, for any given~$k\in\N$.
\end{lemma}
\begin{proof}
Our proof starts with an observation.
If two \bondsep s $\{A,B\}$ and $\{A',B'\}$ cross, then $A'$ contains a vertex from $A$ and a vertex from $B$. 
Let $v\in A'\cap A$ and $w\in A'\cap B$. 
Since $G[A']$ is connected, there exists a path from $v$ to $w$ in $G[A']$. 
This path, and thus $G[A']$, must contain an edge from $A$ to~$B$. Similarly, $G[B']$ must contain an edge from $A$ to $B$.

Now suppose for a contradiction that there are infinitely many \bondsep s of order at most a given $k\in\N$, which all cross some finite-order \bondsep\ $\{A,B\}$. 
Without loss of generality, all the crossing \bondsep s have order~$k$.
Using our observation, the pigeon-hole principle and the finite order of $\{A,B\}$, we find two edges $e,f\in E(A,B)$ and infinitely many \bondsep s $\{A_0,B_0\},\{A_1,B_1\},\ldots\,$ that all cross $\{A,B\}$ so that $e\in G[A_n]$ and $f\in G[B_n]$ for all~$n\in\N$.
Let $P$ be a path in $G$ that links an endvertex $v$ of $e$ to an endvertex $w$ of~$f$.
Now $v$ is contained in all the $A_n$ and $w$ is contained in all the $B_n$, thus for every $\{A_n,B_n\}$ there exists an edge of $P$ with one end in $A_n$ and the other in $B_n$. However, every $\{A_n,B_n\}$ corresponds to a bond of size $k$ of $G$ and, again by the pigeon-hole principle, infinitely many of theses bonds must contain the same edge of~$P$. This contradicts \cref{lem:finitely_many_bonds}.
\end{proof}

Next, to show the second property, we need the following lemma:
\begin{lemma}\label{lem:fish}
If two \cutsep s $\{A_1,B_1\}$ and $\{A_2,B_2\}$ cross, and a third \cutsep\ $\{X,Y\}$ is nested with both $\{A_1,B_1\}$ and $\{A_2,B_2\}$, then $\{X,Y\}$ is nested with $\{A_1\cap A_2,B_1\cup B_2\}$ (provided that this is a \cutsep ).
\end{lemma}
\begin{proof}
As $\{X,Y\}$ is a \cutsep\ that is nested with $\{A_1,B_1\}$ and $\{A_2,B_2\}$, either $X$ or $Y$ is a subset of $B_1$ or $B_2$, in which case it is immediate that $\{X,Y\}$ is nested with $\{A_1\cap A_2, B_1\cup B_2\}$ as desired, or, one of $X$ and $Y$ is a subset of $A_1$ and one of $X$ and $Y$ is a subset of $A_2$. However, since $A_1\cup A_2\neq V(G)$ (as $\{A_1,B_1\}$ and $\{A_2,B_2\}$ cross) it needs to be the case that either $X\subseteq A_1\cap A_2$ or $Y\subseteq A_1\cap A_2$, so in either case $\{X,Y\}$ is nested with $\{A_1\cap A_2, B_1\cup B_2\}$ as desired.
\end{proof}

Using this lemma, we can now show \cref{property:fish}:

\begin{lemma}\label{lem:splinter_2}
If $\{A,B\}\in \cA_i$ and $\{C,D\}\in \cA_j$ cross with $\abs{i}<\abs{j}$, then $\cA_j$ contains some corner of $\{A,B\}$ and $\{C,D\}$ that is nested with $\{A,B\}$.
\end{lemma}
\begin{proof}
Let us denote the two edge-blocks in $j$ as $\beta$ and $\beta'$ so that $\beta\lives C$ and $\beta'\lives D$. 
Since the order of $\{A,B\}$ is less than~$\abs{j}$, we may assume without loss of generality that $\beta,\beta'\lives A$.
We claim that either $\{A\cap C, B\cup D\}$ or $\{A\cap D, B\cup C\}$ is the desired corner in $\cA_j$, and we refer to them as \emph{corner candidates}. 
Both are \cutsep s that distinguish $\beta$~and~$\beta'$, and both are nested with~$\{A,B\}$.
Furthermore, by \cref{lem:fish}, every \cutsep\ that is nested with both $\{A,B\}$ and $\{C,D\}$ is also nested with both corner candidates.
It remains to show that at least one of the two corner candidates has order at most~$\abs{j}$, because then it lies in~$\cA_j$ as desired.

Let us assume for a contradiction that both corner candidates have order greater than~$\abs{j}$.
Then the two inequalities
\begin{align*}
    \abs{E(A\cap C,B\cup D)}+\abs{E(B\cap D,A\cup C)}\leq{}&\abs{E(A,B)}+\abs{E(C,D)}\\
    \text{and }\;\abs{E(A\cap D,B\cup C)}+\abs{E(B\cap C,A\cup D)} \leq{}&\abs{E(A,B)}+\abs{E(C,D)}
\end{align*}
imply
\[
\abs{E(B\cap D,A\cup C)}<\abs{i}\quad\text{and}\quad \abs{E(B\cap C,A\cup D)}<\abs{i}.
\]
Recall that the edge-blocks forming the pair~$i$ are $k$-edge-blocks for some values $k$ greater than~$\abs{i}$.
One of the edge-blocks of the pair~$i$ lives in~$B$, and due to the latter two inequalities, this edge-block must live either in $B\cap D$ or in~$B\cap C$.
But then either $\{B\cap D,A\cup C\}$ or $\{B\cap C,A\cup D\}$ is a \cutsep\ of order less than~$\abs{i}$ that distinguishes the two edge-blocks forming the pair~$i$, contradicting the fact that an order of at least~$\abs{i}$ is required for that.
\end{proof}
Finally, to show the third property, we need the following lemma:
\begin{lemma}\label{lem:cn_submodular}
Let $\{A_1,B_1\}$ and $\{A_2,B_2\}$ be crossing \cutsep s such that both $\{A_1\cap A_2,B_1\cup B_2\}$ and $\{A_1\cup A_2,B_1\cap B_2\}$ are \cutsep s as well.
Then every \cutsep\ that crosses both $\{A_1\cap A_2,B_1\cup B_2\}$ and $\{A_1\cup A_2,B_1\cap B_2\}$ must also cross both $\{A_1,B_1\}$ and $\{A_2,B_2\}$.
\end{lemma}
\begin{proof}
Consider any \cutsep\ $\{X,Y\}$ that crosses both $\{A_1\cap A_2,B_1\cup B_2\}$ and $\{A_1\cup A_2,B_1\cap B_2\}$.
Since $\{X,Y\}$ crosses $\{A_1\cap A_2,B_1\cup B_2\}$, both $X$ and $Y$ contain a vertex from $A_1\cap A_2$. Since  $\{X,Y\}$ crosses  $\{A_1\cup A_2,B_1\cap B_2\}$, both $X$ and $Y$ contain a vertex from $B_1\cap B_2$. 
Hence $\{X,Y\}$ crosses both $\{A_1,B_1\}$ and $\{A_2,B_2\}$.
\end{proof}
Let us now show \cref{property:strong_submodular}:
\begin{lemma}\label{lem:splinter_3}
If $ \{A,B\}\in\cA_i $ and $ \{C,D\}\in\cA_j $ cross with $ \abs{i}=\abs{j}=k \in \N $, then either $ \cA_i $ contains a corner of $  \{A,B\} $ and $ \{C,D\}$ with strictly lower $ k $-crossing number than~$ \{A,B\}$, or else $ \cA_j $ contains a corner of $ \{A,B\}$ and $ \{C,D\}$ with strictly lower~$ k $-crossing number than~$ \{C,D\}$. 
\end{lemma}
\begin{proof}
Let us assume without loss of generality that the $k$-crossing number of $\{A,B\}$ is less than or equal to the $k$-crossing number of $\{C,D\}$, and let us denote the edge-blocks in $j$ as $\beta$ and $\beta'$ so that $\beta\lives C$ and $\beta'\lives D$. 
We consider two cases.

In the first case, $\{A,B\}$ distinguishes the two edge-blocks $\beta$~and~$\beta'$.
Hence $\beta\lives A\cap C$ and $\beta'\lives B\cap D$, say. 
Then both $\{A\cap C, B\cup D\}$ and $\{B\cap D, A\cup C\}$ distinguish the two edge-blocks $\beta$~and~$\beta'$ that form the pair $j$, and so they have order at least $\abs{j}=k$.
Furthermore, we have
\begin{equation}\label{eq:submodular}
    \abs{E(A\cap C,B\cup D)}+\abs{E(B\cap D,A\cup C)}\le\abs{E(A,B)}+\abs{E(C,D)}=2k,
\end{equation}
so both $\{A\cap C,B\cup D\}$ and $\{B\cap D,A\cup C\}$ must have order exactly~$k$.
In particular, both are contained in~$\cA_j$, and they are corners of $\{A,B\}$ and $\{C,D\}$ by~\cref{lem:fish}.
Next, we assert that the $k$-crossing numbers of $\{A\cap C, B\cup D\}$ and $\{B\cap D, A\cup C\}$ in sum are less than the sum of the $k$-crossing numbers of $\{A,B\}$ and $\{C,D\}$.
Indeed, all the $k$-crossing numbers involved are finite by~\cref{property:fin_cn}, and the two \cutsep s $\{A,B\}$ and $\{C,D\}$ cross which allows us to deduce the desired inequality between the sums by Lemmas~\ref{lem:fish} and~\ref{lem:cn_submodular}, as follows: 
\begin{itemize}
    \item by \cref{lem:fish}, every $\{X,Y\}\in\cA$ of order~$k$ that crosses at least one of $\{A\cap C, B\cup D\}$ and $\{B\cap D, A\cup C\}$ must cross at least one of $\{A,B\}$ and $\{C,D\}$; and
    \item by \cref{lem:cn_submodular}, every $\{X,Y\}\in\cA$ of order~$k$ that crosses both $\{A\cap C, B\cup D\}$ and $\{B\cap D, A\cup C\}$ must cross both $\{A,B\}$ and $\{C,D\}$.
\end{itemize}
But then the strict inequality between the sums, plus our initial assumption that the $k$-crossing number of $\{A,B\}$ is less than or equal to that of~$\{C,D\}$, implies that one of $\{A\cap C, B\cup D\}$ and $\{B\cap D, A\cup C\}$ must have a $k$-crossing number less than the one of $\{C,D\}$, as desired.

In the second case, $\{A,B\}$ does not distinguish the two edge-blocks $\beta$~and~$\beta'$.
Recall that all the edge-blocks in the two pairs $i$~and~$j$ are $\ell$-edge-blocks for some values~$\ell>k$.
Hence $\beta\cup \beta'\lives A$, say.
Let us denote by $\beta''$ the edge-block in $i$ that lives in~$B$.
Then either $\beta''\lives B\cap C$ or $\beta''\lives B\cap D$, say $\beta''\lives B\cap D$.
In total:
\[
    \beta\lives A\cap C,\;\, \beta'\lives A\cap D\text{ and }\beta''\lives B\cap D.
\]
Therefore, $\{A\cap C, B\cup D\}$ distinguishes the two edge-blocks $\beta$~and~$\beta'$ forming the pair~$j$ which imposes an order of at least~$k$, and $\{B\cap D, A\cup C\}$ distinguishes the two edge-blocks forming the pair~$i$ which imposes an order of at least~$k$ as well.
Combining these lower bounds with~(\ref{eq:submodular}) we deduce that both $\{A\cap C, B\cup D\}$ and $\{B\cap D, A\cup C\}$ have order exactly~$k$.
In particular, they are contained in $\cA_j$ and $\cA_i$ respectively, and they are corners of $\{A,B\}$ and $\{C,D\}$ by~\cref{lem:fish}.
Repeating the final argument of the first case, we deduce from the strict inequality between the sums of the $k$-crossing numbers that either $\{A\cap C, B\cup D\}\in\cA_j$ has strictly lower $k$-crossing number than~$\{C,D\}$, or else $\{B\cap D, A\cup C\}\in\cA_i$ has strictly lower $k$-crossing number than~$\{A,B\}$, completing the proof.
\end{proof}
We can now prove our main result:
\begin{proof}[Proof of \cref{thm:main}]
Let $G$ be any connected graph.
By \cref{lem:splinter_1}, \cref{lem:splinter_2} and \cref{lem:splinter_3} we may apply \cref{thm:thinly} to the family $(\,\cA_i\mid i\in I\,)$ defined at the beginning of the section. 
This results in the desired nested set $N(G)\subset\cA$. 
To see that it is canonical, note that any isomorphism $\phi\colon G\to G'$ induces an isomorphism between $(\cA,\sim)$ and $(\cA',\sim')$, where the latter is defined like the former but with regard to~$G'$. 
Thus, by the `moreover' part of \cref{thm:thinly}, we indeed obtain that $\phi(N(G))=N(\phi(G))$. 
\end{proof}

\section{Nested sets of bonds and \tcd s}\label{sec:TreePartition}

\noindent Recall that, given a connected graph $G$, we denote by $N=N(G)$ the canonical set of nested bonds from Theorem~\ref{thm:main} that efficiently distinguishes all the edge-blocks of~$G$.
Furthermore, recall that the subset $N_k\subset N$ is formed by the bonds in~$N$ of order less than~$k$.
In this section, we show that:
\begin{itemize}
    \item For every $k\in\N$, the subset $N_k\subset N$ is equal to the set of fundamental cuts of a tree-cut decomposition of $G$ that decomposes $G$ into its $k$-edge-blocks.
\end{itemize}
To this end, we first introduce the notion of a tree-cut decomposition.
Recall that a \emph{near-partition} of a set $M$ is a family of pairwise disjoint subsets $M_\xi\subset M$, possibly empty, such that~$\bigcup_\xi M_\xi=M$.

Let $G$ be a graph, $T$ a tree, and let $\cX=(X_t)_{t\in T}$ be a family of vertex sets $X_t\subset V(G)$ indexed by the nodes $t$ of~$T$.
The pair $(T,\cX)$ is called a \emph{\tcd } of $G$ if $\cX$ is a near-partition of~$V(G)$.
The vertex sets $X_t$ are the \emph{parts} or \emph{bags} of the \tcd\ $(T,\cX)$.
When we say that $(T,\cX)$ \emph{decomposes} $G$ into its $k$-edge-blocks for a given~$k$, we mean that the non-empty parts of $(T,\cX)$ are the sets of vertices of the $k$-edge-blocks of~$G$.
In this paper, we require the nodes with non-empty parts to be \emph{dense} in~$T$ in that every edge of $T$ lies on a path in~$T$ that links up two nodes with non-empty parts.

If $(T,\cX)$ is a \tcd , then every edge $t_1 t_2$ of its \emph{decomposition tree} $T$ induces a cut $E(\,\bigcup_{t\in T_1}X_t\,,\,\bigcup_{t\in T_2}X_t\,)$ of $G$ where $T_1$ and $T_2$ are the two components of $T-t_1 t_2$ with $t_1\in T_1$ and $t_2\in T_2$.
Here, the nodes with non-empty parts densely lying in~$T$ ensures that both unions are non-empty, which is required of the sides of a cut.
We call these induced cuts the \emph{fundamental cuts} of the \tcd ~$(T,\cX)$.
Note that, unlike the fundamental cuts of a spanning tree, the fundamental cuts of a \tcd\ need not be bonds.

It is important that parts of a \tcd\ are allowed to be empty, as the following example demonstrates.

\begin{example}
Let the graph $G$ arise from the disjoint union of three copies $G_1,G_2$ and $G_3$ of $K^4$ by selecting one vertex $v_i\in G_i$ for all $i\in [3]$ and adding all edges $v_i v_j$ ($i\neq j\in [3]$).
Then the $3$-edge-blocks of $G$ are the three vertex sets $V(G_1)$, $V(G_2)$ and $V(G_3)$.
Since $N(G)$ is canonical, we have $N_3(G)=\{\,F_1,F_2,F_3\,\}$ where $F_i:=\{\,v_i v_j\mid j\neq i\,\}$. 
However, we cannot find a \tcd\ $(T,\cX)$ of $G$ such that, on the one hand, $T$ is a tree on three nodes $t_1,t_2,t_3$ and $X_{t_i}=V(G_i)$ for all $i\in [3]$, and on the other hand, the fundamental cuts of $(T,\cX)$ are precisely the bonds in~$N_3(G)$:
the decomposition tree $T$ would then be a path of length two, and hence would induce two fundamental cuts, but $N_3(G)$ consists of three bonds.
\end{example}

To relate $N_k$ to a tree-cut decomposition, we will use a theorem by Gollin and Kneip.
In order to state their theorem, we need to introduce separation systems and $S$-trees first.

\subsection{Separation systems and \texorpdfstring{$\boldsymbol{S}$}{S}-trees}
Separation systems and $S$-trees are two fundamental tools in graph minor theory.
In this section we briefly introduce the definitions from~\cites{AbstractSepSys,DiestelBook5,RhdTreeSets} that we need.

A \emph{separation of a set} $V$ is an unordered pair $\{A,B\}$ such that $A\cup B=V$.
The ordered pairs $(A,B)$ and $(B,A)$ are its \emph{orientations}.
Then the \emph{oriented separations} of $V$ are the orientations of its separations.
The map that sends every oriented separation $(A,B)$ to its \emph{inverse} $(B,A)$ is an involution that reverses the partial ordering
\[
    (A,B)\le (C,D)\;:\Leftrightarrow\;A\subset C\text{ and }B\supset D
\]
since $(A,B)\le (C,D)$ is equivalent to $(D,C)\le (B,A)$.

More generally, a \emph{separation system} is a triple $(\vS,{\le},{}^\ast)$ where $(\vS,{\le})$ is a partially ordered set and ${}^\ast\colon\vS\to\vS$ is an order-reversing involution.
We refer to the elements of $\vS$ as \emph{oriented separations}.
If~an oriented separation is denoted by $\vs$, then we denote its \emph{inverse} $\vs^\ast$ as $\sv$, and vice versa.
That ${}^\ast$ is \emph{order-reversing} means $\vr\le\vs\Leftrightarrow\rv\ge\sv$ for all $\vr,\vs\in\vS$.

A \emph{separation} is an unordered pair of the form $\{\vs,\sv\}$, and then denoted by $s$.
Its elements $\vs$ and $\sv$ are the \emph{orientations} of $s$.
The set of all separations $\{\vs,\sv\}\subset\vS$ is denoted by $S$.
When a separation is introduced as $s$ without specifying its elements first, we use $\vs$ and $\sv$ (arbitrarily) to refer to these elements.

Separations of sets, and their orientations, are an instance of this abstract setup if we identify $\{A,B\}$ with $\{\,(A,B)\,,(B,A)\,\}$.
Hence the \cutsep s of a graph define a separation system.
Here is another example:
The set $\vE(T):=\{\,(x,y)\mid xy\in E(T)\,\}$ of all \emph{orientations} $(x,y)$ of the edges $xy=\{x,y\}$ of a tree $T$ forms a separation system with the involution $(x,y)\mapsto (y,x)$ and the natural partial ordering on $\vE(T)$ in which $(x,y)<(u,v)$ if and only if $xy\neq uv$ and the unique $\{x,y\}$--$\{u,v\}$ path in $T$ is $\mathring{x}yTu\mathring{v}=yTu$. 
\begin{figure}[ht]
\centering
\includegraphics[width=.3\textwidth]{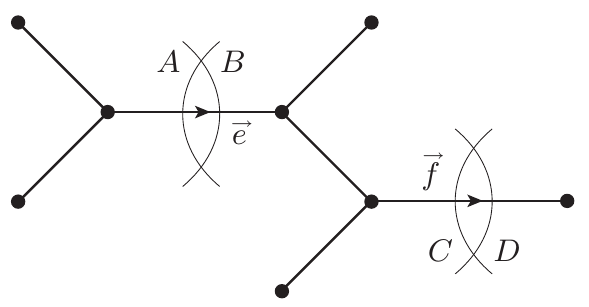}
\caption{An $S$-tree with $\alpha(\ve)=(A,B)\le (C,D)=\alpha(\vf)$.~\cite{DiestelBook5}}
\label{fig:Stree}
\end{figure}

An $S$-\emph{tree} is a pair $(T,\alpha)$ such that $T$ is a tree and $\alpha\colon\vE(T)\to\vS$ propagates the ordering on $\vE(T)$ and commutes with inversion: that $\alpha(\ve)\le\alpha(\vf)$ if $\ve\le\vf\in\vE(T)$ and $(\alpha(\ev))^\ast=\alpha(\ve)$ for all $\ve\in \vE(T)$; see Figure~\ref{fig:Stree}.
A tree-decomposition $(T,\cV)$, for example, makes $T$ into an $S$-tree for the set of separations it induces~\cite{DiestelBook5}*{§12.5}.
Similarly, a tree-cut decomposition $(T,\cX)$ makes $T$ into an $S$-tree for the set of \cutsep s which correspond to its fundamental cuts.

An \emph{isomorphism} between two separation systems is a bijection between their underlying sets that respects both their partial orderings and their involutions.
We need the following fragment of \cite{Kneip}*{Theorem~1} by Gollin and Kneip:

\begin{theorem}\label{thm:Kneip}
Let $G$ be any connected graph, and let $\vS$ be any nested separation system formed by oriented \cutsep s of~$G$.
Then the following assertions are equivalent:
\begin{enumerate}
    \item There exists an $S$-tree $(T,\alpha)$ such that $\alpha\colon\vE(T)\to\vS$ is an isomorphism between separation systems;
    \item $\vS$ contains no chain of order-type~$\omega+1$.
\end{enumerate}
\end{theorem}

\subsection{\texorpdfstring{$\boldsymbol{N_k}$}{Nk} is a set of fundamental cuts}

The following theorem clearly implies that $N_k$ is the set of fundamental cuts of a tree-cut decomposition of $G$ that decomposes $G$ into its $k$-edge-blocks:

\begin{theorem}\label{thm:generalTreePartition}
Let $G$ be any connected graph and $k\in\N$.
Every nested set of bonds of $G$ of order less than~$k$ is the set of fundamental cuts of some \tcd\ of~$G$.
\end{theorem}

\begin{proof}
Let $G$ be any connected graph, $k\in\N$, and let $B$ be any nested set of bonds of $G$ of order less than~$k$.
We write $S$ for the set of \bondsep s which correspond to the bonds in~$B$.

First, we wish to use Theorem~\ref{thm:Kneip} to find an $S$-tree $(T,\alpha)$ such that $\alpha\colon\vE(T)\to\vS$ is an isomorphism.
For this, it suffices to show that $B$ cannot contain pairwise distinct bonds $F_0,F_1,\ldots,F_\omega$ such that each bond $F_\alpha$ has a side $A_\alpha$ with $A_\alpha\subsetneq A_\beta$ for all $\alpha<\beta\le\omega$.
This is immediate from Corollary~\ref{cor:noOmegaOrPlusOne}.

Second, we wish to find a \tcd\ $(T,\cX)$ whose fundamental cuts are precisely equal to the bonds in~$B$.
We define the parts $X_t$ of $(T,\cX)$ by letting \[X_t:=\medcap\,\{\,D\mid (C,D)=\alpha(x,t)\text{ where }xt\in E(T)\,\}.\]
Then clearly the parts $X_t$ are pairwise disjoint.
To see that $\bigcup_t X_t$ includes the whole vertex set of $G$, consider any vertex $v\in V(G)$.
We orient each edge $t_1 t_2\in T$ towards the $t_i$ with $v\in D$ for $(C,D)=\alpha(t_{3-i},t_i)$.
By~Corollary~\ref{cor:noOmegaOrPlusOne} we may let $t$ be the last node of a maximal directed path in~$T$; then all the edges of $T$ at $t$ are oriented towards~$t$, and $v\in X_t$ follows.
Therefore, $\cX$ is a near-partition of~$V(G)$.
It is straightforward to see that $B$ is the set of fundamental cuts of~$(T,\cX)$.
\end{proof}

\subsection{\texorpdfstring{$\boldsymbol{N}$}{N} is not a set of fundamental cuts}

Finally, we show that there exists a graph $G$ that has no nested set of cuts which, on the one hand, distinguishes all the edge-blocks of~$G$ efficiently, and on the other hand, is the set of fundamental cuts of some tree-cut decomposition.

\begin{figure}[ht]
    \includegraphics[width=.5\linewidth]{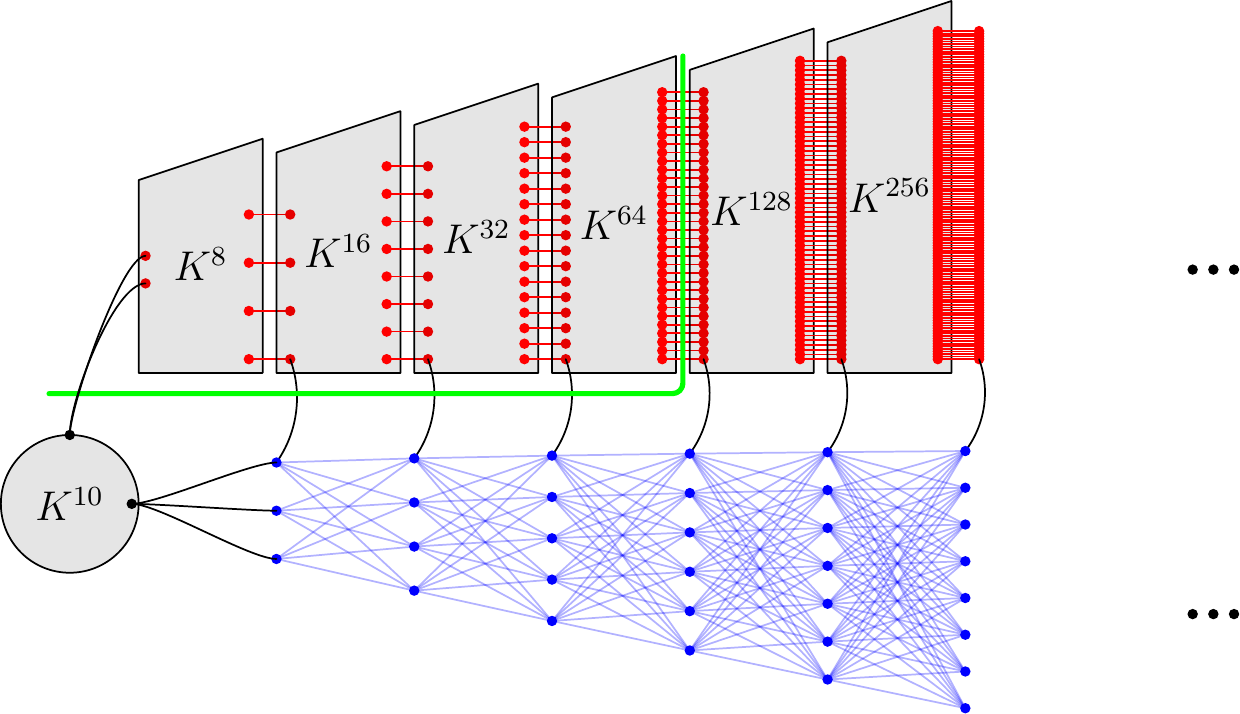}
    \caption{The only cut that efficiently distinguishes the two edge-blocks defined by $K^{64}$ and by $K^{128}$ is drawn in green.}
    \label{fig:no_efficient_td}
\end{figure}
\begin{example}\label{example:BigPicture}
This example is a variation of \cite{InfiniteSplinters}*{Example 4.9}.
Consider the locally finite graph displayed in Figure~\ref{fig:no_efficient_td}. This graph $G$ is constructed as follows.
For every $n\in \N_{\ge 1}$ we pick a copy of $K^{2^{n+2}}$ together with $n+2$ additional vertices $w_1^n,\dots,w_{n+2}^n$. Then we select $2^{n}$ vertices of the $K^{2^{n+2}}$ and call them $u_1^n,\dots,u_{2^n}^n$. 
Furthermore, we select $2^{n+1}$ vertices of the  $K^{2^{n+2}}$, other than the previously chosen $u_i^n$, and call them $v_1^n,\dots,v_{2^{n+1}}^n$. 
Now we add all the red edges $v_i^n u_i^{n+1}$, all the blue edges $w_i^n w_j^{n+1}$, and if $n\ge 2$ we also add the black edge $u_1^n w_1^n$.
Finally, we disjointly add one copy of $K^{10}$ and join one vertex $v_1^0$ of this $K^{10}$ to $u_1^1$ and $u_2^1$; and we select another vertex $w_1^0\in K^{10}$ distinct from $v_1^0$ and add all edges $w_1^0 w_i ^1$.
This completes the construction.

Now the vertex sets of the chosen $K^{2^{n+2}}$ are $(2^{n+2} - 1)$-edge-blocks~$B_n$. 
The only \cutsep\ that efficiently distinguishes $B_n$ and $B_{n+1}$ is $F_n:=\{\,\bigcup_{k=1}^n B_n\,,\,V\setminus \bigcup_{k=1}^n B_n\,\}$.
Additionally, the vertex set of the $K^{10}$ is a $9$-edge-block~$B_0$. The only \cutsep\ that efficiently distinguishes $B_0$ and $B_1$ is $F_0:=\{B_0,V\setminus B_0\}$.
Therefore, $N(G)$ must contain all the cuts corresponding to the \cutsep s~$F_n$~($n\in\N$).
But the \cutsep s $F_n$ define an $(\omega+1)$-chain
\[
    (B_1,V\setminus B_1)<(B_1\cup B_2,V\setminus (B_1\cup B_2))<\cdots<(V\setminus B_0,B_0),
\]
so $N(G)$ cannot be equal to the set of fundamental cuts of a tree cut-decomposition of~$G$ by Theorem~\ref{thm:Kneip}. 
\end{example}

\section{Generating all bonds}\label{sec:DD}

\noindent A set $S$ of \cutsep s \emph{generates} a cut $\{X,Y\}$ if there exists a finite subset $\{\,\{A_k,B_k\}\mid k<n\,\}\subset S$ such that
\[
    \{X,Y\}=\{\,\medcup_{k<n}A_k\,,\, \medcap_{k<n}B_k\,\}.
\]
If $S$ generates $\{X,Y\}$, then the cuts corresponding to the \cutsep s in $S$ \emph{generate} the cut corresponding to~$\{X,Y\}$.
Note that $S$ generates $\{X,Y\}$ if and only if both $(X,Y)$ and $(Y,X)$ can be obtained from finitely many oriented \cutsep s in~$\vS$ by taking suprema and infima, where 
\begin{itemize}
    \item $(A,B)\vee (A',B'):=(A\cup A',B\cap B')$ is the \emph{supremum} and
    \item $(A,B)\wedge (A',B'):=(A\cap A',B\cup B')$ is the \emph{infimum}
\end{itemize}
of two \cutsep s $(A,B)$ and $(A',B')$.
In this section we prove our second main result:

\begin{customthm}{\ref{thm:generate}}
Let $G$ be any connected graph and let $M$ be any nested set of bonds of~$G$. 
Then the following assertions are equivalent:
\begin{enumerate}
    \item $M$ efficiently distinguishes all the edge-blocks of~$G$;
    \item For every $k\in\N$, the ${\le} k$-sized bonds in~$M$ generate all the $k$-sized cuts of~$G$.
\end{enumerate}
\end{customthm}

\noindent For the proof, we need a generalised version of the star-comb lemma~\cite{DiestelBook5}*{Lemma~8.2.2}.
A \emph{comb} in a given graph $G$ means one of the following two substructures of~$G$:
\begin{enumerate}
    \item The union of a ray $R$ (the comb's \emph{spine}) with infinitely many disjoint finite paths, possibly trivial, that have precisely their first vertex on~$R$. 
The last vertices of those paths are the \emph{teeth} of this comb.
\item The union of a ray $R$ (the comb's \emph{spine}) with infinitely many disjoint pairwise inequivalent rays $R_0,R_1,\ldots$ that have precisely their first vertex on~$R$.
The ends to which the rays $R_0,R_1,\ldots$ belong are the \emph{teeth} of this comb.
\end{enumerate}
Given a set~$U\subset V(G)\cup\Omega(G)$, a \emph{comb attached to} $U$ is a comb with all its teeth in $U$.
A \emph{star attached to}~$U$ is either a subdivided infinite star with all its leaves in $U$, or a union of infinitely many rays that meet precisely in their first vertex and belong to distinct ends in~$U$.

\begin{lemma}[Generalised star-comb lemma]
Let $U\subset V(G)\cup\Omega(G)$ be an infinite set for a connected graph $G$.
Then $G$ contains either a comb \at $U$ or a star \at $U$.
\end{lemma}

\begin{proof}
If $U$ contains infinitely many vertices of~$G$, then we are done by the standard star-comb lemma~\cite{DiestelBook5}*{Lemma~8.2.2}.
Hence we may assume that $U$ consists of ends and, say, is countable.
Inductively, we choose for each end $\omega\in U$ a ray $R_\omega\in\omega$ so that $R_\omega$ is disjoint from all previously chosen rays, ensuring that all chosen rays are pairwise disjoint, and we let $U'$ consist of the first vertices of these rays.
Then we consider an inclusionwise minimal tree $T\subset G$ that extends all the rays $R_\omega$ with $\omega\in U$.
Let $T'\subset T$ be the inclusionwise minimal subtree that contains~$U'$.
Then, by the standard star-comb lemma, $T'$ contains either a star or a comb attached to $U'$, and either extends to a star or comb attached to~$U$.
\end{proof}

\noindent For more on stars and combs, see the series~\cites{StarComb1StarsAndCombs,StarComb2TheDominatedComb,StarComb3TheUndominatedComb,StarComb4TheUndominatingStar}.

\begin{proof}[Proof of Theorem~\ref{thm:generate}]
(ii)$\to$(i) 
Let $M$ be any nested set of bonds of~$G$ such that, for every $k\in\N$, the ${\le} k$-sized bonds in~$M$ generate all the $k$-sized cuts of~$G$, and suppose for a contradiction that there are two edge-blocks $\beta_1,\beta_2$ which are not efficiently distinguished by any bond in~$M$. 
Let $\{X,Y\}$ be some \bondsep\ which efficiently distinguishes $\beta_1$ and $\beta_2$, and let $k$ be its order. 
Let $\{\,\{A_\ell,B_\ell\}\mid \ell<n\,\}$ be a finite set of ${\le} k$-sized bonds in $M$ which generate $\{X,Y\}$ so that $\{X,Y\}=\{\,\bigcup_{\ell<n}A_\ell\,,\,\bigcap_{\ell<n}B_\ell\,\}$. 
Since $M$ does not efficiently distinguish $\beta_1$ from $\beta_2$, for every $\ell<n$ we either have that both $\beta_1$ and $\beta_2$ live in~$A_\ell$, or that both of them live in $B_\ell$. However, this implies that either both $\beta_1$ and $\beta_2$ live in $X$, or that both of them live in $Y$, contradicting the fact that $\{X,Y\}$ distinguishes $\beta_1$ and $\beta_2$.

(i)$\to$(ii) We assume~(i).
It suffices to prove (ii) for finite bonds.
Let $B=E(V_1,V_2)$ be any bond of~$G$ of size~$k$, say.
By Theorem~\ref{thm:generalTreePartition}, the set formed by the ${\le}k$-sized bonds in~$M$ is the set of fundamental cuts of a tree-cut decomposition~$(T,\cX)$ of~$G$.
Write $(T,\alpha)$ for the $S$-tree that arises from~$(T,\cX)$.

Since $B$ is finite, only finitely many parts of $(T,\cX)$ contain endvertices of edges in~$B$. 
We let $H$ be the minimal subtree of $T$ which contains all the nodes corresponding to these parts.
Note that $H$ is finite.
Then we let $H'$ be the subtree of $T$ which is induced by the nodes of $H$ and all their neighbours in~$T$.
The subtree $H'$ might be infinite, but it is rayless.
Let $\cH$ be the tree-cut decomposition of $G$ which corresponds to the $S$-tree $(H',\,\alpha\rest\vE(H'))$.

We claim that every two edge-blocks of $G$ that are distinguished by $B$ are also distinguished by some fundamental cut of $\cH$.
For this, let $\beta_1\lives V_1$ and $\beta_2\lives V_2$ be any two edge-blocks of~$G$ that are distinguished by~$B$.
Then $\beta_1$ and $\beta_2$ are also distinguished by a ${\le k}$-sized bond in~$M$, and hence some fundamental cut of $(T,\cX)$ distinguishes $\beta_1$ and $\beta_2$ as well.
Let $st$ be an edge of $T$ whose induced fundamental cut distinguishes $\beta_1$ and $\beta_2$, chosen at minimal distance to $H'$ in~$T$.
Then $\beta_1$ lives in $C$ and $\beta_2$ lives in $D$ for $(C,D)=\alpha(s,t)$, say.
We claim that $st$ is also an edge of~$H'$, and assume for a contradiction that it is not.
Then~$s$, say, is not a vertex of $H'$ and $t$ lies on the $s$--$H'$ path in~$T$.
Since $\{C,D\}$ is an element of~$M$, it is a bond and in particular $G[C]$ is connected.
Moreover, $C$ avoids the endvertices of the edges in~$B$, because $t$ separates $s$ from $H$.
Therefore, $C$ is included in one of the two sides of~$B$, say in~$V_1$, so $\beta_1$ lives in~$V_1$.
The node $t$, however, cannot lie in $H$ because this would imply $s\in H'$, so $t$ has a neighbour $u$ in $T$ which separates $t$ (and $s$) from~$H$.
Let $(C',D'):=\alpha(t,u)$.
Since $s$ and $u$ are distinct neighbours of~$t$, we have $(C,D)\le (C',D')$.
As argued for $(C,D)$, we find that $C'$ must be included in one of the two sides of~$B$, and this side must be $V_1$ since $C$ is included in both $V_1$ and $C'$.
By the choice of $st$ at minimal distance to~$H'$, the edge-block $\beta_2$ must live in~$C'$ (or we could replace $st$ with~$tu$, contradicting the choice of~$st$).
But then both $\beta_1$ and $\beta_2$ live in~$V_1$, the desired contradiction.

We replace $(T,\cX)$ with $\cH$. Then:
\medskip
\begin{fleqn}%
\begin{equation*}%
\tag{$\ast$}%
\hspace{2\parindent}\begin{aligned}\label{Bdist}%
    \parbox{\textwidth-5\parindent}{Every two edge-blocks of $G$ that are distinguished by $B$ are also distinguished by some fundamental cut of~$(T,\cX)$.}%
\end{aligned}%
\end{equation*}%
\end{fleqn}%

Given a node $t\in T$, we denote by $\hat{X}_t$ the subset of $\hat{V}(G)$ 
which is the union of all the $(k+1)$-edge-blocks of $G$ that live in $D$ for all \cutsep s $(C,D)=\alpha(s,t)$ with $(s,t)\in\vE(T)$.
Then $\hat{X}_t\cap V(G)=X_t$ and we call $\hat{X}_t$ the \emph{extended part} of $t$.
Note that extended parts of distinct nodes are disjoint.
Since $T$ is rayless, the extended parts near-partition~$\hat{V}(G)$.
As an immediate consequence of~(\ref{Bdist}), every extended part of $(T,\cX)$ lives either in~$V_1$ or $V_2$.

We colour the nodes of~$T$ using red and blue, as follows.
We colour a node $t\in T$ red if $\hat{X}_t$ is non-empty and $\hat{X}_t\lives V_1$.
Similarly, we colour a node $t\in T$ blue if $\hat X_t$ is non-empty and $\hat X_t\lives V_2$. 
Finally, we consider all the nodes $t\in \hat T$ with $\hat X_t=\emptyset$.
These induce a forest in~$T$.
We colour all the nodes in a component of this forest red if the component has a red neighbour, and blue otherwise.

We let $T_1\subset T$ be the forest induced by the red nodes, and we let $T_2\subset T$ be the forest induced by the blue nodes.
The way in which we coloured the nodes with empty extended parts ensures that, for every connected component $C$ of~$T_1$ or of~$T_2$, 
 some node $t\in C$ has a non-empty extended part $\hat X_t$.
Note that $B=E(\,\bigcup_{t\in T_1} X_t\,,\,\bigcup_{t\in T_2} X_t\,)$ by the definition of $T_1$~and~$T_2$.
We claim that we are done if $T$ contains only finitely many $T_1$--$T_2$ edges.
Indeed, if $s_0 t_0,\ldots, s_n t_n$ are the finitely many $T_1$--$T_2$ edges with $s_\ell\in T_1$ and $t_\ell\in T_2$, then
\begin{align*}
    (V_1,V_2)=\;\bigwedge_{C\text{: a component of }T_2}\quad\bigvee_{\ell\text{: }t_\ell\in C} \;\;\alpha(s_\ell,t_\ell)\;.
\end{align*}
Thus, it remains to show that $T$ contains only finitely many $T_1$--$T_2$ edges.
For this, we consider the tree $\tilde{T}$ that arises from $T$ by contracting every component of $T_1$ and every component of~$T_2$ to a single node. 
Since $T$ is rayless, so is~$\tilde{T}$.
By K\H{o}nig's lemma, it remains to show that $\tilde{T}$ is locally finite.

Suppose for a contradiction that $d\in\tilde{T}$ is a vertex that has some infinitely many neighbours $c_n$ ($n\in\N$).
Recall that all the sets $Y_c:=\bigcup\,\{\,\hat{X_t}\mid t\in c\,\}$ where $c$ is a node of $\tilde{T}$ are non-empty.
We choose one point $u_n\in Y_{c_n}$ for every~$n\in\N$, and we apply the star-comb lemma in the connected side~$G[V_i]$ of~$B$ where all sets~$Y_{c_n}$ live to the infinite set $U:=\{\,u_n\mid n\in\N\,\}$.
Then we cannot get a star, because the finite fundamental cuts of $(T,\cX)$ induced by its $T_i$--$d$ edges would force the centre vertex to lie in~$Y_d$, contradicting the fact that $Y_d$ lives in $V_{3-i}$.
Therefore, the star-comb lemma must return a comb contained in~$G[V_i]$ and attached to~$U$. Without loss of generality, each $u_n$ is a tooth of this comb.

Let us consider the end of~$G$ that contains the spine of the comb.
This end is contained in a $(k+1)$-edge-block $\beta\lives V_i$.
And $\beta$ in turn is included in a set $Y_c$ where $c$ is a component of~$T_i$.
Hence $c\neq d$.
But then the fundamental cut of $(T,\cX)$ which corresponds to the $T_i$--$d$ edge on the $c$--$d$ path in $T$ separates a tail of the comb from infinitely many~$u_n$, a contradiction.
\end{proof}

\section{Finitely separating spanning trees and \texorpdfstring{$\infty$}{infinity}-edge-blocks}\label{sec:infEdgeBlocks}

\noindent By the second property of our nested set $N(G)$, we find a tree-cut decomposition of any connected graph~$G$ into its $k$-edge-blocks, one for every $k\in\N$.
But for $k=\infty$, such a decomposition does not in general exist, e.g., consider Example~\ref{example:BigPicture} with each $K^n$ of the graph replaced by $K^{\aleph_0}$ (or any other infinitely edge-connected graph).
The reason why, however, is not that there are no meaningful tree-cut decompositions of $G$ into its $\infty$-edge-blocks, but that we considered only those decompositions whose sets of fundamental cuts are equal to~$N(G)$.
If we drop this requirement, then we find tree-cut decompositions of $G$ into its $\infty$-edge-blocks, meaningful in the sense that all their fundamental cuts are finite.
Let us call a graph \emph{finitely separable} if any two of its vertices can be separated by finitely many edges.
And let us call a spanning tree, respectively a tree-cut decomposition, \emph{finitely separating} if all its fundamental cuts are finite.
The following theorem has been introduced in~\cite{StarComb3TheUndominatedComb} as Theorem~3.9, and it is Theorem~5.1 in~\cite{TypicalInfinitelyEdgeconnectedGraphs}:

\begin{theorem}[\,\cite{StarComb3TheUndominatedComb}\,]
Every finitely separable connected graph has a finitely separating spanning tree.
\end{theorem}
\noindent If $G$ is any connected graph, then the graph $\tilde{G}$ obtained from $G$ by collapsing every $\infty$-edge-block to a single vertex is finitely separable and connected.
Hence $\tilde{G}$ has a finitely separating spanning tree by the theorem, and this tree is easily translated to a finitely separating tree-cut decomposition of~$G$, even with all parts non-empty:
\begin{theorem}[\,\cite{TypicalInfinitelyEdgeconnectedGraphs}\,]
Every connected graph has a finitely separating tree-cut decomposition into its \mbox{$\infty$-edge-blocks}.\qed
\end{theorem}
\noindent This result, phrased in terms of $S$-trees, is extensively used in~\cite{TypicalInfinitelyEdgeconnectedGraphs} to study infinite edge-connectivity.

\input{EdgeBlockBib.tex}
\end{document}

%% file: Preamble.tex
\usepackage[utf8]{inputenc} %test
\usepackage{amssymb}
\usepackage{mathrsfs}
\usepackage{graphicx}
\usepackage{latexsym}
\usepackage{stmaryrd}
\usepackage{enumitem}
\usepackage[bookmarks,hyperfootnotes=false]{hyperref}
\usepackage{multirow}
\usepackage{xcolor}
\usepackage{caption}
\colorlet{darkishRed}{red!80!black}
\colorlet{darkishBlue}{blue!60!black}
\colorlet{darkishGreen}{green!60!black}
\hypersetup{
    draft = false,
    bookmarksopen=true,
    colorlinks,
    linkcolor={red!60!black},
    citecolor={green!60!black},
    urlcolor={blue!60!black}
}
\usepackage[capitalize]{cleveref}
\usepackage[msc-links]{amsrefs} 
\usepackage{doi}

\renewcommand{\PrintDOI}[1]{\doi{#1}}

\let\setminus=\smallsetminus
\usepackage{tikz}
\usepackage{tikz-cd}
\usetikzlibrary{calc,through,intersections,arrows, trees, positioning, decorations.pathmorphing, cd}
\usepackage{relsize}
\usepackage{comment}
\usepackage{svg}
\usepackage{mathtools}
\usepackage{nccmath}
\usepackage{pifont}

\usepackage{geometry}
\let\setminus=\smallsetminus

\newcommand{\rest}{\upharpoonright}

\DeclareMathOperator{\medcup}{\mathsmaller{\bigcup}}

\DeclareMathOperator{\medcap}{\mathsmaller{\bigcap}}
\newcommand{\lives}{\sqsubseteq}

\renewcommand{\subset}{\subseteq}
\renewcommand{\supset}{\supseteq}

\newcommand{\at}{attached to }

\newcommand{ \N } { \mathbb{N} }

\makeatletter

\def\calCommandfactory#1{%
   \expandafter\def\csname c#1\endcsname{\mathcal{#1}}}
\def\frakCommandfactory#1{%
   \expandafter\def\csname frak#1\endcsname{\mathfrak{#1}}}

\newcounter{ctr}
\loop
  \stepcounter{ctr}
  \edef\X{\@Alph\c@ctr}
  \expandafter\calCommandfactory\X
  \expandafter\frakCommandfactory\X
  \edef\Y{\@alph\c@ctr}
  \expandafter\frakCommandfactory\Y
\ifnum\thectr<26
\repeat

\setenumerate{label={\normalfont (\roman*)}}%,itemsep=0pt}

\usepackage[presets={vec-cev,abc,ABC,cAcBcC}]{letterswitharrows}

\newtheorem{theorem}{Theorem}[section] 
\newtheorem{corollary}[theorem]{Corollary}
\newtheorem{lemma}[theorem]{Lemma}

\newtheorem{mainresult}{Theorem}

\newtheorem{innercustomthm}{Theorem}

\newenvironment{customthm}[1]
  {\renewcommand\theinnercustomthm{#1}\innercustomthm}
  {\endinnercustomthm}

\theoremstyle{definition}
\newtheorem{example}[theorem]{Example}

\theoremstyle{remark}